\documentclass[leqno]{amsart}
\usepackage{stmaryrd,graphicx}
\usepackage{amssymb,mathrsfs,amsmath,amscd,amsthm}
\usepackage{float}
\usepackage[all,cmtip]{xy}
\usepackage{amsfonts,txfonts,pxfonts,latexsym,wasysym}
\DeclareMathAlphabet{\mathpzc}{OT1}{pzc}{m}{it}

\newtheorem{theorem}{Theorem}[section]
\newtheorem{lemma}[theorem]{Lemma}
\newtheorem{proposition}[theorem]{Proposition}
\theoremstyle{definition}
\newtheorem{definition}[theorem]{Definition}
\newtheorem{problem}[theorem]{Problem}
\newtheorem{remark}[theorem]{Remark}
\numberwithin{equation}{section}
\newtheorem{corollary}[theorem]{Corollary}
\newtheorem{example}[theorem]{Example}

\begin{document}
\title[Strongly Pseudoradial Spaces]{Strongly Pseudoradial Spaces}
\author{Jeremy Brazas and Paul Fabel}

\begin{abstract}
The ``weakly Hausdorff" property for pseudoradial spaces fails to be
naturally characterized by unique convergence of transfinite sequences. In
response, we develop the category $\mathbf{SPsRad}$ of strongly pseudoradial
spaces, compactly generated spaces whose closed sets are determined by
globally continuous maps from well-ordered spaces. Categorically $\mathbf{%
SPsRad}$ is the coreflective hull of the class of well-ordered spaces, and $%
\mathbf{SPsRad}$ is Cartesian closed. The strongly pseudoradial weakly
Hausdorff spaces admit a natural characterization involving unique
extensions of injective maps of well-ordered spaces. We also obtain analogs
in $\mathbf{SPsRad}$ of the fact that for sequential spaces, sequential
compactness is equivalent to countable compactness.
\end{abstract}

\maketitle

\section{\textbf{Introduction}}

This paper introduces a category \textbf{SPsRad }of\textbf{\ strongly
pseudoradial spaces, }a natural generalization of sequential spaces. The
space $X$ is strongly pseudoradial if for each nonclosed set $A\subset X,$
there exists (with the order topology), a noncompact well ordered space $%
\alpha $\textbf{\ }and a map $f:\alpha \cup \{\infty _{\alpha }\}\rightarrow
X$ so that $f(\alpha )\subset A$ and $f(\infty _{\alpha })\notin A.$ Here $%
\alpha \cup \{\infty _{\alpha }\}$ denotes the familiar one point
compactification of $\alpha $.

Our motivation is the observation that the following fact about sequential
spaces does not (as shown in Example \ref{kcnou}) generalize naturally to
pseudoradial spaces. If $X$ is a sequential space then each convergent
sequence in $X$ has a unique limit if and only if each compact subspace of $%
X $ is closed. On the other hand, this fact generalizes naturally to
strongly pseudoradial spaces (Corollary \ref{sutc}).

Recall the space $X$ is \textbf{sequential} \cite{Franklin1}\cite{Franklin2}
if for each nonclosed $A\subset X$ there exists a convergent sequence $%
a_{n}\rightarrow x$ so that $\{a_{n}\}\subset A$ and $x\notin A,$ $X$ is 
\textbf{compactly generated} (\textbf{CG}) \cite{Steenrod}\cite{Strickland}
if for each nonclosed $A\subset X$ there exists a compact Hausdorff space $K$
and a map $f:K\rightarrow X$ such that $f(K)\cap A$ is not closed in $X,$
and $X$ is \textbf{pseudoradial} (\textbf{PsRad}) \cite{Arc}\cite{Howes}\cite%
{Tironi} if for each nonclosed $A\subset X$ there exists an unbounded
well-ordered set $\alpha ,$ a function $f:\alpha \cup \{\infty _{\alpha
}\}\rightarrow A$ so that $f$ is continuous at $\infty _{\alpha },$ $%
f(A)\subset X,$ and $f(\infty _{\alpha })\notin A.$ Thus each sequential
space is strongly pseudoradial, and every strongly pseudoradial space is
both pseudoradial and compactly generated. Example \ref{strict1} shows a
compactly generated pseudoradial space need not be strongly pseudoradial.

If the sequential space $X$ is $T_{1},$ then nonclosed sets are detected by
convergent sequences of distinct points. The natural generalization (Theorem %
\ref{sprchar}) in \textbf{SPsRad }employs injective maps of regular
cardinals with the order topology.

Recall the space $X$ is \textbf{US} if convergent sequences in $X$ have
unique limits, $X$ is \textbf{WH} (i.e. weakly Hausdorff), if maps of
compact $T_{2}$ spaces into $X$ have closed image \cite{Steenrod}\cite%
{Strickland}, and $X$ is a \textbf{KC space} if compact subsets of $X$ are
closed. For general spaces $T_{2}\Rightarrow KC\Rightarrow WH\Rightarrow
US\Rightarrow T_{1},$ and the implications are strict \cite{Wil}. If $X$ is 
\textbf{CG}, then $WH\Longleftrightarrow KC$ and if $X$ is sequential then $%
US\Longleftrightarrow WH\Longleftrightarrow KC.$ For strongly pseudoradial
spaces, the natural generalization of the $US$ property is \textbf{unique
strong pseudoradial convergence:} for each map $f:\alpha \rightarrow X$ of a
non-compact well-ordered space $\alpha ,$ there is at most one continuous
extension to $\alpha \cup \{\infty _{\alpha }\}.$ As noted, we show a
strongly pseudoradial space $X$ is $KC$ if and only if it has unique strong
pseudoradial convergence. This follows from Theorem \ref{closedmap} which
asserts that maps of compact well-ordered spaces into such spaces $X$ have
closed image. Example 3 shows this result does not translate naturally to
the traditional pseudoradial category.

In contrast with the fact that the category $\mathbf{PsRad}$ of pseudoradial
spaces fails to be Cartesian closed \cite{Cincura}, the Convenient
categorical properties (in the sense of \cite{Steenrod}) of $\mathbf{SPsRad}$
are established in section \ref{cat}. The examples in section \ref{examples}
illustrate various ways in which the main results are best possible, and
also establish strict relationships among various categories under
consideration.

Theorem \ref{compactthm} is an analog (or generalization modulo Lemma \ref%
{adjust}) of the fact that for sequential spaces, countable compactness is
equivalent to sequential compactness. However in $\mathbf{SPsRad}$ the
relationship between compactness and related notions quickly leads to deep
waters in axiomatic set theory, even for first countable countably compact
spaces \cite{EisworthRoitman}\cite{Eisworth}\cite{EisworthNyikos}\cite{OSt},
(and much more generally for pseudoradial spaces \cite{Dow}\cite{Dow1}\cite%
{Dow2} \cite{OSt}\cite{Shap}\cite{Sleziak}\cite{Tironi}.) For example, in $%
\mathbf{SPsRad}$ a natural analog of ``sequentially compact'' is \textbf{%
strongly pseudoradially compact} ($\mathbf{SPC}$). The space $X$ has the
latter property if for each noncompact well-ordered space $\alpha $ and each
map $f:\alpha \rightarrow X,$ there exists $\beta \subset \alpha ,$ closed
and cofinal, so that $f|\beta $ is continuously extendable to the one point
compactification $\beta \cup \{\infty _{\beta }\}.$ By inspection this
notion is equivalent to compactness for well ordered spaces, but in general
the notions are inequivalent, and plausibly unrelated. Assuming Jensen's
combinatorial diamond principle $\diamondsuit ,$ Ostaszewski spaces \cite%
{OSt}\cite{Nyikos} exist, and are noncompact strongly pseudoradial spaces
satisfying $\mathbf{SPC}$. On the other hand in $\mathbf{SPsRad}$ every
compact weakly Hausdorff space is $\mathbf{SPC}$ (Corollary \ref{compactcor}%
). The authors were unable to settle whether all compact, $T_{1}$ $\mathbf{%
SPsRad}$ spaces are $\mathbf{SPC}$ (Problem \ref{t1prob}).

\section{Definitions and preliminaries}

By a \textbf{well-ordered space} $X$, we mean a well-ordered set with the 
\textbf{order topology} generated by half open intervals $(a,b]=\{x\in
X|a<x\leq b\}$ (and by closed intervals $[m,b]$ if $m$ is minimal in $X$).

If $\alpha $ is a noncompact well-ordered space $\alpha \cup \{\infty
_{\alpha }\}$ denotes the (unique) well-ordered one point compactification
of $\alpha ,$ (i.e. the compact space obtained by attaching a maximal point $%
\infty _{\alpha }$ to $\alpha $ whose basic neighborhoods are of the form $%
(a,\infty _{\alpha }].)$

\begin{remark}
\emph{\ If $\alpha $ is a noncompact well-ordered space and $\beta \subset
\alpha $ then the subspace topology of $\beta $ coincides with the order
topology of $\beta $ iff $\beta $ is a compact subspace of $\alpha \cup
\{\infty _{\alpha }\}$ or (if $\beta $ is not compact and taking closure in $%
\alpha \cup \{\infty _{\alpha }\}$), sup$\beta =\overline{\beta }\backslash
\beta .$}
\end{remark}

If $\alpha $ is a well-ordered set a subset $K\subset \alpha $ is \textbf{%
cofinal} if for each $x\in \alpha $ there exists $k\in K$ so that $x\leq k.$ 
$K$ is an \textbf{initial segment} if $K=\alpha $ or $K=[0,x)=\{y\in \alpha
|0\leq y<x\}$ for some $x\in \alpha .$

If $A$ and $B$ are sets then $\left\vert A\right\vert <\left\vert
B\right\vert $ means there exists an injection from $A$ into $B$ and no
injection is surjective, and $\left\vert A\right\vert =\left\vert
B\right\vert $ means there exists a bijection from $A$ onto $B.$ An \textbf{%
ordinal }is a well-ordered set. A \textbf{cardinal} $\alpha $ is a
well-ordered set so that $\left\vert \beta \right\vert <\left\vert \alpha
\right\vert $ for every proper initial segment $\beta .$ A \textbf{regular
cardinal} $\alpha $ is a cardinal such that if $\beta \subset \alpha $ and $%
\beta $ is cofinal then there is exists an order preserving bijection from $%
\beta $ onto $\alpha .$ (We are formally ignoring the standard notion
`cofinality').

If $A$ and $B$ are subsets of the linearly ordered set $(S,<)$ the notation $%
A<B$ means $a<b$ for all $a\in A$ and all $b\in B.$

By the \textbf{directed topology} on the well-ordered set $\alpha \cup
\{\infty _{\alpha }\}$ we mean the (generally finer) space with topology
generated by sets $\{a\}$ and $(a,\infty _{\alpha }]$, with $a\in \alpha .$
The space $X$ is \textbf{strongly pseudoradial} if for each nonclosed $%
A\subset X$ there exists a noncompact well-ordered space $\alpha $ and a map 
$f:\alpha \cup \{\infty _{\alpha }\}\rightarrow X$ so that $f(\alpha
)\subset A$ and $f(\infty _{\alpha })\notin A.$ If the same conclusion holds
with respect to the directed topology on $\alpha \cup \{\infty _{\alpha }\}$%
, we obtain the generally weaker property that $X$ is \textbf{pseudoradial. }%
The space $X$ has \textbf{unique strong pseudoradial convergence} if $%
f(\infty _{\alpha })=g(\infty _{\alpha })$ for all noncompact well-ordered
spaces $\alpha ,$ and all pairs of maps $f,g:\alpha \cup \{\infty _{\alpha
}\}\rightarrow X$ such that $f|\alpha =g|\alpha .$ If the same conclusion
holds with respect to the directed topology on $\alpha \cup \{\infty
_{\alpha }\}$ we obtain the generally stronger property that $X$ has \textbf{%
unique pseudoradial convergence.}

A space $X$ is \textbf{sequentially compact} if each sequence in $X$ has a
convergent subsequence and $X$ is \textbf{countably compact} if each
countable open cover of $X$ has a finite subcover.

\section{Characterizing strongly pseudoradial spaces}

The main result of this section (Corollary \ref{big}) generalizes to
strongly pseudoradial spaces the following fact about sequential spaces: the
space $X$ is sequential iff for each nonclosed $A\subset X$ there exists $%
a\in A$ such that $\overline{\{a\}}\backslash A\neq \emptyset $ or there
exists a convergent sequence of \textit{distinct} points $a_{n}\rightarrow b$
so that $\{a_{1},a_{2},...\}\subset A$ and $b\notin A$ (i.e. there exists a
continuously extendable injective map $f:\alpha \rightarrow A$ of the
regular cardinal $\omega =\{1,2,3...\},$ so that $f(\infty _{\omega })\notin
A).$

The following Lemmas are well known or straightforward. While likely
apparent to the reader proficient in elementary set theory \cite{Hrb}, care
is needed to ensure extra topological or structural conditions are met. For
example, in Lemma \ref{mono} merely invoking the axiom of choice does not
guarantee the useful extra properties that $\beta $ is closed or cofinal.

\begin{lemma}
\label{mono}Suppose $\alpha $ is a noncompact well-ordered space and $%
\{A_{i}\}$ is a partition of $\alpha $ into bounded sets (indexed by a set $%
I $). Then there exists a closed cofinal subspace $\beta \subset \alpha $ so
that $\left| A_{i}\cap \beta \right| \leq 1$ for all $i\in I.$
\end{lemma}

\begin{proof}
If $x\in \alpha $ let $A_{i_{x}}$ denote the (unique) element of $\{A_{i}\}$
so that $x\in A_{i_{x}}.$ Let $0$ denote the minimal element of $\alpha .$
Let $\beta _{0}=\{0\}.$

Suppose $x\in \alpha \cup \{\infty _{\alpha }\}$ and $\beta _{a}\subset
\alpha $ has been defined for all $a<x$. Let $\gamma _{x}=\cup _{a<x}\beta
_{a}.$ Let $m_{x}=\sup \gamma _{x}$ in the space $\alpha \cup \{\infty
_{\alpha }\}.$ Suppose the following three conditions hold: i) $x\leq m_{x}$
ii) if $\{a,b\}\subset \gamma _{x}$ and $a<b$ then $A_{i_{a}}<\{b\},$ iii) $%
\gamma _{x}\cup \{m_{x}\}$ is compact.

Observe conditions (i) (ii) and (iii) are true in case $x=0.$ If $x>0$
proceed as follows.

Case 1. If $m_{x}=\infty _{\alpha }$ let $\beta =\gamma _{x}$ and the
theorem is proved by conditions (ii) and (iii) and the fact that $\gamma
_{x} $ is unbounded in $\alpha .$

Case 2. Suppose $m_{x}<\infty _{\alpha }.$ Obtain $r_{x}\in \alpha $ minimal
so that $A_{i_{m_{x}}}<r_{x}.$ Define $\beta _{x}=\gamma _{x}\cup
\{m_{x}\}\cup \{r_{x}\}.$ Observe $\beta _{x}$ is compact since $\gamma
_{x}\cup \{m_{x}\}$ is compact (by the induction hypothesis). Note $m_{x}\in
A_{i_{m_{x}}}$ and hence $m_{x}<r_{x}.$ Note $\gamma _{x+1}=\beta _{x}$ and $%
m_{x+1}=r_{x}.$

We now check conditions (i) to (iii) are preserved at the index $x+1.$ To
check iii) recall $\gamma _{x+1}\cup \{m_{x+1}\}=\beta _{x},$ which is
compact, as shown above. To check i) recall $m_{x+1}=r_{x}$ and $m_{x}<r_{x}$
as shown above. By hypothesis $x\leq m_{x}$ and thus $x+1\leq m_{x}+1\leq
r_{x}=m_{x+1}$. To check ii) suppose $\{a,b\}\subset \gamma _{x+1}$ and $%
a<b. $ Recall $\gamma _{x}\leq m_{x}<r_{x}.$ If $\gamma _{x}$ is compact
then $m_{x}\in \gamma _{x}.$ If $b\in \gamma _{x}$ then $a\in \gamma _{x}$
and $A_{i_{a}}<b$ by the induction hypothesis. If $b=m_{x}$ and $\gamma _{x}$
is not compact obtain $c$ so that $a<c<b$ and thus by the induction
hypothesis $A_{i_{a}}<c<b.$ Suppose $b=r_{x}.$ If $a<m_{x}$ then $%
A_{i_{a}}<m_{x}<b$ and if $a=m_{x}$ then $A_{i_{m_{x}}}<r_{x}$ by definition
of $r_{x}.$ In case $x=\infty _{\alpha }$ we achieve Case 1.
\end{proof}

\begin{lemma}
\label{setbij}Suppose $\alpha $ and $J$ are unbounded well-ordered sets and $%
h:\alpha \rightarrow J$ is a bijection. There exists a subset $K\subset
\alpha $ so that $h|K$ is order preserving and $h(K)$ is cofinal in $J.$
\end{lemma}

\begin{proof}
Manufacture $K$ as follows. Let $k_{1}$ denote the minimal element of $%
\alpha .$ Let $K_{1}=\{k_{1}\}.$ Observe $h|K_{1}$ is order preserving and $%
K_{1}<h^{-1}(y)$ if $h(K_{1})<\{y\}.$ Observe $[k_{1},\infty )\cap K_{1}\neq
\emptyset .$ Suppose $i\in \alpha \cup \infty $ and $K_{j}$ is defined for
each $j<i.$ Suppose $h|(\cup _{j<i}K_{j})$ is order preserving and $\cup
_{j<i}K_{j}<h^{-1}(y)$ if $h(\cup _{j<i}K_{j})<\{y\}.$ Suppose $[j,\infty
)\cap K_{j}\neq \emptyset $ for each $j<i.$ Case 1. If $h(\cup _{j<i}K_{j})$
is cofinal let $K=\cup _{j<i}K_{j}$ and observe the Lemma is proved. Case 2.
If $h(\cup _{j<i}K_{j})$ is not cofinal, by hypothesis there exists $%
k_{i}\in \alpha $ minimal so that $\cup _{j<i}K_{j}<\{k_{i}\}$ and $h(\cup
_{j<i}K_{j})<\{h(k_{i})\}.$ Let $K_{i}=\cup _{j<i}K_{j}\cup \{k_{i}\}.$
Observe $h|K_{i}$ is order preserving. Minimality of $k_{i}$ ensures $%
K_{i}<h^{-1}(y)$ if $h(K_{i})<\{y\}.$ By the induction hypothesis $j<k_{i}$
for each $j<i.$ Thus $i\leq k_{i}.$ Case 1 is reached by the time $\alpha
=\infty .$
\end{proof}

\begin{lemma}
\label{ezreg}If $\alpha $ is an unbounded well-ordered set and $\beta $ has
minimal order type among ordinals $\beta \subset \alpha $ such that $\beta $
is cofinal in $\alpha $ then $\beta $ is a regular cardinal.
\end{lemma}

\begin{proof}
Given well-ordered sets $\gamma $ and $\beta $ we say $\gamma \prec \beta $
if $\gamma $ can be embedded as a proper initial segment of $\beta $ and $%
\gamma \preceq \beta $ if $\gamma \prec \beta $ or $\gamma $ is isomorphic
to $\beta .$ To obtain a contradiction obtain $\alpha $ of minimal order
type so that the result fails.

If $\beta $ $\prec \alpha $ then obtain a regular cardinal $\gamma $ cofinal
in $\beta .$ Observe $\gamma \preceq \beta $ and $\gamma $ is cofinal in $%
\alpha $. By minimality of $\beta $ we conclude $\gamma $ is isomorphic to $%
\beta $ and thus $\beta $ is a regular cardinal, a contradiction.

By definition if $\beta $ is isomorphic to $\alpha $ and $\alpha $ is a
cardinal then $\alpha $ is a regular cardinal. If $\alpha $ is not a
cardinal obtain $x\in \alpha $ minimal so that $\left\vert [0,x)\right\vert
=\left\vert \alpha \right\vert .$ Obtain a bijection $h:[0,x)\rightarrow
\alpha $ and apply Lemma \ref{setbij} to obtain $K\subset \lbrack 0,x)$ so
that $h|K$ is order preserving and cofinal in $\alpha .$ Since $K\prec
\alpha $ and $\alpha $ is a minimal counterexample, obtain a

regular cardinal $\gamma $ cofinal in $K$ and note $h|\gamma $ is cofinal in 
$\alpha $ and we have a contradiction since $\gamma \prec \alpha .$
\end{proof}

We can easily ensure that cofinal sets are closed in the order topology.

\begin{lemma}
\label{reg}Suppose $\alpha $ is a well-ordered space. There exists a closed
cofinal subspace $\beta \subset \alpha $ so that $\beta $ is a regular
cardinal.
\end{lemma}

\begin{proof}
Obtain by Lemma \ref{ezreg} a (possibly non closed) cofinal $\gamma \subset
\alpha $ so that $\gamma $ is a regular cardinal. Let $\partial \gamma =%
\overline{\gamma }\backslash \gamma .$ For each $x\in \partial \gamma $
obtain $k_{x}\in \gamma $ minimal so that $x<k_{x}.$ Let $K$ denote the
union of such points $k_{x}.$ Let $\beta =\overline{\gamma }\backslash K.$
Note $K$ is open in the space $\alpha $ and hence $\beta $ is closed. The
order preserving bijection $\beta \rightarrow \gamma $ fixing $\gamma
\backslash K$ pointwise and sending $x\in \partial \gamma $ to $k_{x}\in K$
shows $\beta $ has the same order type as $\gamma ,$ and thus $\beta $ is a
regular cardinal. Since $\gamma $ is cofinal in $\alpha $, $\gamma
\backslash K$ is cofinal in $\alpha $ (since for each $k_{x}\in K$ the next
element of $\gamma $ is not in $K.)$ Thus $\beta $ is cofinal in $\alpha .$
\end{proof}

In similar manner to $T_{1}$ sequential spaces, we can detect nonclosed sets
in a $T_{1}$ strongly pseudoradial space with extendable injective maps of
regular cardinals.

\begin{theorem}
\label{sprchar}A space $X$ is strongly pseudoradial iff for each nonclosed $%
A\subset X$ there exists $a\in A$ such that $\overline{\{a\}}\backslash
A\neq \emptyset $ or there exists (with the order topology) a regular
cardinal $\gamma $ and a continuous injection $\kappa :\gamma \cup \{\infty
_{\gamma }\}\rightarrow X$ such that $\kappa (\gamma )\subset A$ and $\kappa
(\infty _{\gamma })\notin A.$
\end{theorem}

\begin{proof}
Suppose $X$ is strongly pseudoradial and $A\subset X$ is not closed. Suppose
there does not exist $a\in A$ and $b\notin A$ so that $b\in \overline{\{a\}}%
. $ Obtain a noncompact ordinal $\beta $ and a map $f:$ $\beta \cup \{\infty
_{\beta }\}\rightarrow X$ such that $f(\beta )\subset A$ and $f(\infty
_{\beta })\notin A.$ For $a\in A$ let $S_{a}=f^{-1}(a)\subset \beta .$ Note $%
S_{a}$ is bounded since otherwise continuity of $f$ at $\infty _{\beta }$
shows $f(\infty _{\beta })$ is a limit point of the singleton $\{a\}.$ Thus
the sets $\{S_{a}\}$ form a partition of $\alpha $ into bounded sets. By
Lemma \ref{mono} there exists a closed cofinal set $\alpha \subset \beta $
such that $f|\alpha $ is one to one. By Lemma \ref{reg} there exists a
closed and cofinal regular ordinal $\gamma \subset \alpha $. Let $\kappa
=f|\gamma .$

For the converse suppose $A\subset X$ is not closed. If there exists $a\in A$
and $b\notin A$ such that $b\in \overline{\{a\}}$ employ the map of $\omega
\cup \{\infty _{\omega }\}$ sending $n\rightarrow a$ and $\infty _{\omega
}\rightarrow b.$ If no such $a$ exists, employ the advertised map $\kappa $
and conclude that $X$ is strongly pseudoradial.
\end{proof}

As one might expect, $\left| X\right| $ is the lower bound on the size of
cardinals needed to detect nonclosed sets in the $T_{1}$ strongly
pseudoradial space $X.$

\begin{corollary}
\label{big}Suppose $X$ is a space and $\alpha $ is the cardinal such that $%
\left| X\right| =\left| \alpha \right| .$ Then $X$ is strongly pseudoradial
iff for each nonclosed $A\subset X$ there exists $a\in A$ such that $%
\overline{\{a\}}\backslash A\neq \emptyset $ or there exists (with the order
topology) a regular cardinal $\beta $ (with $\left| \beta \right| \leq
\left| \alpha \right| $) and a continuous injection $h:\beta \cup \{\infty
_{\beta }\}\rightarrow X$ such that $h(\beta )\subset A$ and $h(\infty
_{\beta })\notin A.$
\end{corollary}

\section{Unique strong pseudoradial convergence}

A sequential space $X$ is a $US$ space iff $X$ is a $KC$ space. Corollary %
\ref{sutc} provides a strong analogue in the \textbf{SPsRad} category.
Example \ref{kcnou} shows the natural analogue fails in the \textbf{PsRad}
category.

\begin{theorem}
\label{closedmap}Suppose $X$ is strongly pseudoradial and has unique strong
pseudoradial convergence. Suppose $\alpha $ is a noncompact well-ordered
space and $h:\alpha \cup \{\infty _{\alpha }\}\rightarrow X$ is a continuous
injection. Then $h(\alpha \cup \{\infty _{\alpha }\})$ is closed in $X.$
\end{theorem}

\begin{proof}
To obtain a contradiction suppose the claim is false. Obtain an ordinal of
minimal order type of the form $\alpha \cup \{\infty _{\alpha }\}$ such that 
$Y=h(\alpha \cup \{\infty _{\alpha }\})$ is not closed in $X,$ but so that $%
h([0,a])$ is closed in $X$ for all $a\in \alpha .$ Since $X$ is $T_{1}$ and
strongly pseudoradial, and since $Y$ is not closed in $X,$ obtain a
noncompact ordinal $\beta $ and a continuous injection $g:\beta \cup
\{\infty _{\beta }\}\rightarrow X$ so that $g(\beta )\subset Y$ and $%
g(\infty _{\beta })\notin Y$. If $h(\infty _{\alpha })\notin g(\beta )$ let $%
\gamma =\beta .$ If $h(\infty _{\alpha })\in g(\beta )$ let $\gamma =\beta
\backslash \lbrack 0_{\beta },g^{-1}h(\infty _{\alpha })]$. Define $\kappa
:\gamma \cup \{\infty _{\beta }\}\rightarrow Y$ so that $\kappa |\gamma
=g|\beta $ and define $\kappa (\infty _{\beta })=h(\infty _{\alpha }).$ Note
the injection $g|(\gamma \cup \{\infty _{\beta }\})$ is continuous. Thus,
since $X$ has strong unique transfinite convergence, $\kappa $ is not
continuous. Since $\kappa |\gamma $ is continuous, $\kappa $ is not
continuous at the point $\infty _{\beta }.$ Thus there exists an open set $%
U\subset X$ so that $h(\infty _{\alpha })\in U$ and $(b,\infty _{\beta
}]\backslash \kappa ^{-1}(U)\neq \emptyset $ for all $b\in \gamma .$ Since $%
\kappa (\infty _{\beta })\in U,$ $(b,\infty _{\beta }]\backslash \kappa
^{-1}(U)=(b,\infty _{\beta })\backslash \kappa ^{-1}(U).$ Hence, since $%
\kappa |\gamma $ is continuous, $(b,\infty _{\beta })\backslash \kappa
^{-1}(U)$ is a nonempty open subspace of $\gamma $ for all $b\in \gamma .$
Let $K=\gamma \backslash \kappa ^{-1}(U).$ Observe $K$ is a noncompact
closed subspace of $\gamma ,\kappa (K)\cap U=\emptyset ,$ and $K\cup
\{\infty _{\beta }\}$ is compact. Since $h$ is continuous at $\infty
_{\alpha },$ there exists $a\in \alpha $ so that $h(i)\in U$ if $a<i.$ Hence 
$\kappa (K)\subset h([0_{\alpha },a]).$ By hypothesis $h([0_{\alpha },a])$
is closed in $X.$ The injective map $g|(K\cup \{\infty _{\beta }\})$ shows $%
g(\infty _{\beta })$ is a limit point of $g(K)$ and thus $g(\infty _{\beta
}) $ is in the closed set $h([0,a]).$ This contradicts the fact that $%
g(\infty _{\beta })\notin Y=h(\alpha \cup \{\infty _{\alpha }\}).$
\end{proof}

\begin{corollary}
\label{sutc}Suppose $X$ is strongly pseudoradial. The following are
equivalent:
\end{corollary}

\begin{enumerate}
\item $X$ has unique strong pseudoradial convergence,

\item If $C$ is a compact well-ordered space and $h:C\rightarrow X$ is a
continuous injection then $h$ is a closed embedding,

\item $X$ is weakly Hausdorff,

\item $X$ is a $KC$ space.
\end{enumerate}

\begin{proof}
By definition (4) $\Rightarrow$ (3) $\Rightarrow$ (2) for all spaces $X$. To show (2) $\Rightarrow$ (4) suppose $A\subset X$ is not closed. By Theorems \ref{sprchar}
and \ref{closedmap} obtain a noncompact ordinal $\alpha $ and a closed
embedding $\kappa :\alpha \cup \{\infty _{\alpha }\}\rightarrow X$ so that $%
\kappa (\alpha )\subset A$ and $\kappa (\infty _{\alpha })\notin A.$ Since $%
\kappa $ is a closed map, $\kappa (\alpha )=im(\kappa )\cap A$ is a closed
subspace of the space $A.$ Since the closed subspace $\kappa (\alpha )$ of $%
A $ is not compact, $A$ is not a compact space.

(1) $\Rightarrow$ (2) follows directly from Theorem \ref{closedmap}. If (1) is
false obtain a noncompact ordinal $\alpha $ and an extendable continuous
injection $f:\alpha \rightarrow X$ with nonunique extensions $\infty
_{\alpha }\rightarrow x$ and $\infty _{\alpha }\rightarrow y$ with $x\neq y.$
Then $y$ is a limit point of $f(\alpha )\cup \{x\}$ and hence (2) fails. Thus (2) $\Rightarrow$ (1).
\end{proof}

\section{\label{examples}Examples and counterexamples}

The examples in this section illustrate various ways in which the main
results of this paper are best possible, and also establish $\mathbf{SPsRad}$
is a proper subcategory of $\mathbf{PsRad}\cap \mathbf{CG}.$

The $1\Rightarrow 2$ implication of Corollary \ref{sutc} yields closed
embeddings for pseudoradial spaces. Given 1, if we drop the pseudoradial
hypothesis we could hope in principle to obtain embeddings that are not
necessarily closed. The following example, obtained by attaching an extra
point to the Arens-Fort space shows this is generally hopeless.

\begin{example}
\emph{\ There exists a $US$ space $Y$ so that with the order topology on $%
\omega\cup \{\infty_{\omega} \},$ there exists a continuous bijection $%
h:\omega\cup \{\infty_{\omega} \}\rightarrow Y$ so that $h$ is \textbf{not}
a homeomorphism.}
\end{example}

\begin{proof}
The well-known Arens-Fort space \cite[Example 26]{Steen} is a countable space with a single non-isolated point, such that no sequence of isolated points converges to the non-isolated point. In particular, we may take the Arens-Fort space to be the natural numbers $X=\{1,2,3....\}$ so that each
single-point set $\{2\},\{3\},....$ is open, $\{1\}$ is a limit point of $\{2,3,4....\}$,
and such that no sequence in $\{2,3,4,...\}$ converges to $1.$ Let $Y=X\cup \{\infty \}$ be the space obtained by attaching an extra point to $X$ with basic open sets $\{\infty,n,n+1,....\}$ at the added point. Note that $Y$ is a non-Hausdorff, US space. The natural identity function $id:\omega \cup \{\infty _{\omega}\}\rightarrow Y$ is a continuous bijection but not a homeomorphism.
\end{proof}

Corollary \ref{big} requires that the cardinal $\alpha $ satisfies $\left|
\alpha \right| =\left| X\right| ,$ and this cannot be relaxed, even if $%
\alpha $ is a non-regular cardinal, as shown in the following example.

\begin{example}
\emph{\ Let $\alpha $ be a non-regular cardinal with the order topology (for
example let $\alpha =\cup \alpha _{n},$ the union of nested regular
cardinals $\alpha _{1}<\alpha _{2}<...$). Let $X=\alpha \cup \{\infty
_{\alpha }\}.$ Observe there is no limit to the length of proper initial
segments of $\alpha $ needed to detect nonclosed sets $A\subset X.$ Thus $%
\alpha $ is the minimal cardinal which makes Corollary \ref{big} true,
despite the fact that no nonclosed set $A\subset X$ requires $\alpha \cup
\{\infty _{\alpha }\}$ to detect the failure of $A$ to be closed.}
\end{example}

The following example illustrates the failure of the pseudoradial analogue
of Corollary \ref{sutc} in the $\mathbf{PsRad}$ category.

\begin{example}
\emph{\label{kcnou}Let }$\alpha $\emph{\ be an uncountable, unbounded
well-ordered set with the discrete topology and attach two unrelated maximal
points. In particular, let $X=\alpha \cup \{x,y\}$ with $x\neq y,$ declare $%
\alpha <x$ and $\alpha <y$ and let $(a,x]$ and $(a,y]$ be basic open sets
for $a\in \alpha .$ Then $X$ is a pseudoradial KC space but $X$ does \textit{%
not} have unique pseudoradial convergence.}
\end{example}

The following example shows $\mathbf{SPsRad}$ is a proper subcategory of $%
\mathbf{PsRad}\cap \mathbf{CG} $

\begin{example}
\emph{\label{strict1} Let $X$ be an uncountable, well-ordered set with the
discrete topology and let $Y=X\cup \{\infty \}$ denote the Alexandroff
compactification of $X.$ Then $Y$ is a compact Hausdorff pseudoradial space,
but $Y$ is not strongly pseudoradial.}
\end{example}

\begin{proof}
Since $Y$ is a compact Hausdorff space, $A$ is not closed in $Y$ iff $A$ is
not compact. Thus $id:Y\rightarrow Y$ shows $Y$ is compactly generated.

To see that $Y$ is pseudoradial, suppose $A\subset Y$ and $A$ is not closed.
Then $\infty $ is a limit point of $A.$ Notice $A\cup \infty $ is a
well-ordered set. Let $A^{\prime }\cup \infty ^{\prime }=A\cup \infty $ with
the directed topology and consider the inclusion function $j:A^{\prime }\cup
\infty ^{\prime }\rightarrow Y\cup \infty $. If $U\subset Y$ is an open set
such that $\infty \in U,$ then $Y\backslash U$ is finite. Select $y\in Y$
such that $Y\backslash U<\{y\}.$ Hence if $y<a$ and $a\in A$ then $j(a)\in
U. $ Thus $Y$ is transfinite sequential.

To see that $Y$ is not strongly pseudoradial , suppose $\alpha $ is a
noncompact well-ordered space and suppose $f:\alpha \cup \{\infty _{\alpha
}\}\rightarrow Y$ is a map such that $f(\alpha )\subset A.$ Note if $i\in
\alpha $ then $f([0,i])$ is a compact subset of the discrete space $A,$ and
hence $f([0,i])$ is finite.

To obtain a contradiction suppose for each integer $N\geq 1$ there exists $%
i_{N}\in \alpha $ such that $\left| f([0,i_{N}])\right| =N.$ Then $%
i_{1}<i_{2}......$

Since $X$ is uncountable there exists a limit $i\in \alpha $. Then $\left|
f([0,i])\right| >N$ for each $N$. However $\left| f([0,i])\right| $ is
finite and we have a contradiction. Since $f$ is continuous at $\infty
_{\alpha }$ we conclude $f(\infty _{\alpha })\in f(\alpha )$ and hence $Y$
is not strongly pseudoradial.
\end{proof}

Every strongly pseudoradial space is compactly generated, however,
pseudoradial spaces need not be compactly generated, as illustrated in the
following example.

\begin{example}
\emph{\label{tfnotcg} $(\mathbf{\mathbf{PsRad}}\backslash \mathbf{CG}\neq
\emptyset )$ Suppose $\alpha $ is a minimal, uncountable well-ordered set
and $X=\alpha \cup \{\infty _{\alpha }\}$ with the directed topology
generated by sets $\{\beta \}$ with $\beta \in \alpha $ and $(\beta ,\infty
_{\alpha }]$. Then $X$ is pseudoradial but is not compactly generated. }
\end{example}

\begin{proof}
Suppose $A\subset X$ and $A$ is not closed in $X$. Then $\overline{A}%
\backslash A\neq \emptyset $ and hence $\{\infty _{\alpha}\}=\overline{A}%
\backslash A$ since $\infty _{\alpha}$ is the only limit point of $X.$

To see that $X\in \mathbf{PsRad}$ let $Y=A\cup \{\infty _{\alpha}\}$ with
the subspace topology. Note $Y$ is a well-ordered set, $A$ is unbounded in $%
Y,$ and $Y$ enjoys the directed topology. Thus inclusion $id:Y\rightarrow X$
is continuous, $id(A)\subset A$ and $id(\infty _{\alpha})\notin A.$

To see that $X\notin \mathbf{CG}$, note $\alpha $ is not closed in $X.$
Suppose $K$ is a compact $T_{2}$ space and $f:K\rightarrow X$ is a map. Let $%
B=f^{-1}(\alpha ).$ It suffices to prove $B$ is closed in $K.$ Note $f(K)$
is compact. To obtain a contradiction suppose $f(K)$ is infinite. Obtain $%
j_{1}$ minimal in $f(K)$ and note $j_{1}\in \alpha .$ Suppose $n>1$ and $%
j_{1}<j_{2}...<j_{n-1}$ have been selected such that $j_{i}\in \alpha \cap
f(K).$ Obtain $j_{n}\in f(K)$ minimal such that $j_{n-1}<j_{n}.$ Note $%
j_{n}\in \alpha .$ Let $C=\{j_{1},j_{2},...\}.$ Since $\alpha $ is
uncountable there exists $l\in \alpha $ such that $C<\{l\}.$ Note $f(K)$ is
Hausdorff since $X$ is Hausdorff, and note $C$ is a closed subspace of $X.$
Hence $C$ is a closed subspace of $f(K)$ and thus $C$ is compact. However $C$
is an infinite space with the discrete topology, and the thus $C$ is not
compact, and we have a contradiction. Thus $f(K)$ is a finite set. Let $%
f(K)\backslash \{\infty _{\alpha }\}=\{k_{1},k_{2},..,k_{n}\}.$ Since $X$ is
Hausdorff, $\{k_{i}\}$ is closed for each $i.$ Thus, since $f$ is
continuous, $B$ is the union of finitely many closed sets $f^{-1}\{k_{i}\}.$
Hence $f^{-1}(\alpha )$ is closed. Thus $X\notin \mathbf{CG}$.
\end{proof}

\begin{lemma}
\label{nots}Suppose $X$ is a Hausdorff space and there exists a countably
infinite set $A=\{a_{1},a_{2},..\}\subset X$ such that $A$ is not closed,
and such that each convergent sequence in $A$ is eventually constant. Then $%
X $ is not pseudoradial.
\end{lemma}

\begin{proof}
Let $\alpha$ be an unbounded, well-ordered set and $\alpha\cup \{\infty
_{\alpha}\}$ has the directed topology. To obtain a contradiction suppose $%
f:\alpha\cup \{\infty _{\alpha}\}\rightarrow X$ is a map such that $%
f(\alpha)\subset A$ and $f(\infty _{\alpha})\notin A.$ Let $%
S_{n}=f^{-1}(a_{n})$ and let $k_{n}=\sup S_{n}.$ Note the singleton $%
f(S_{n})\subset A$. Hence, since $X$ is a $T_{2}$ space, $f(\overline{S_{n}}%
)=f(S_{n}).$ Thus $k_{n}<\infty _{\alpha}$ for all $n$ (since otherwise $%
f(\infty _{\alpha})\in A$). Let $s_{n}=\max \{k_{1},..,k_{n}\}.$ Note $%
s_{1}\leq s_{2}....$ and $s_{n}<\infty _{\alpha}.$ Let $s=\sup \{s_{n}\}$
and note $s=\infty _{\alpha}$ (since otherwise we obtain the contradiction $%
f((s,\infty _{\alpha})))\cap A=\emptyset $). By construction, the sequence $%
s_{n}\rightarrow \infty _{\alpha}.$ By continuity of $f,$ $%
f(s_{n})\rightarrow f(\infty _{v}).$ Thus $\{f(s_{n})\}$ is a convergent
sequence in $A$ and hence $\{f(s_{n})\}$ is eventually constant. Thus $%
f(\infty _{\alpha})\in A$ and we have a contradiction.
\end{proof}

The following example shows $\mathbf{CG}\backslash \mathbf{PsRad}\neq
\emptyset $

\begin{example}
\label{cgnottf} \emph{Let $\omega $ denote the natural numbers with the
discrete topology and let $X=\beta \omega ,$ the Stone-\v{C}ech
compactification of $\omega$ \cite{Engelking}. Then $X$ is a compactly
generated space but is not pseudoradial.}
\end{example}

\begin{proof}
Since $\beta \omega $ is compact Hausdorff, it is compactly generated. Let $%
A=\omega \subset X.$ By construction, $\omega $ is not closed in $\beta
\omega $. Each bounded real valued function $f:\omega \rightarrow \mathbb{R}$
is the restriction of a map $f:\beta \omega \rightarrow \mathbb{R}.$ Thus
each convergent sequence in $\omega $ is eventually constant. Otherwise,
there exists a convergent subsequence $n_{1}<n_{2}.....$ with $%
n_{i}\rightarrow x\in X$. Since $X$ is Hausdorff, the bounded real valued
map such that $f(n_{2i})=0$ and $f(m)=1$ if $m\neq n_{2i}$ cannot be
continuously extended to $\{x,\omega \}$ and we have a contradiction. Now
apply Lemma \ref{nots} to see that $\beta\omega$ is not pseudoradial.
\end{proof}

\section{\label{cat}On the category of strongly pseudoradial spaces}

Let $\mathbf{SPsRad}$ denote the full subcategory of the usual category $%
\mathbf{Top}$ of topological spaces and continuous functions whose objects
are strongly pseudoradial.

Recall that if $\mathscr{A}$ is any class of topological spaces, the \textit{%
coreflective hull} of $\mathscr{A}$ is the subcategory $CH(\mathscr{A})$ of $%
\mathbf{Top}$ whose objects are the spaces homeomorphic to a quotient of a
topological sum of the objects in $\mathscr{A}$. For instance, the
sequential category $\mathbf{Seq}$ is the coreflective hull of the singleton 
$\{\omega+1\}$ and the compactly generated category $\mathbf{CG}$ is the
coreflective hull of the class of all compact Hausdorff spaces.

\begin{proposition}
If $\mathscr{S}$ is the class of well-ordered spaces of the form $\alpha\cup
\{\infty_{\alpha}\}$ where $\alpha$ is a non-compact well-ordered space,
then $CH(\mathscr{S})=\mathbf{SPsRad}$.
\end{proposition}

\begin{proof}
It suffices to show that a space $X$ is strongly pseudoradial if and only if 
$X$ is the quotient of a topological sum of elements of $\mathscr{S}$.

If $X$ is strongly pseudoradial, then for each non-closed set $A\subset X$,
there is a non-compact ordinal $\beta_A$ and a map $f_A:\beta_A\cup
\{\infty_{\beta_A}\}\to X$ such that $f(\beta_A)\subset A$ and $%
f(\infty_{\beta_A})\notin A$. It is straightforward to see that $X$ is the
quotient of $\coprod_{A}\left(\beta_A\cup \{\infty_{\beta_A}\}\right)$ using
the maps $f_A$. The other direction follows from the fact that every
quotient of a strongly pseudoradial space is strongly pseudoradial. The
proof of this fact is the same as the well-known proof that the quotient of
a sequential space is sequential.
\end{proof}

\begin{corollary}
If $\mathscr{O}$ is the class of well-ordered spaces, then $CH(\mathscr{O})=%
\mathbf{SPsRad}$.
\end{corollary}

\begin{proof}
Certainly $\mathscr{S}\subset \mathscr{O}$, thus $CH(\mathscr{S})\subseteq
CH(\mathscr{O})$. Recall that every well-ordered space $\beta$ is a retract
of an element of $\mathscr{S}$ (see \cite[Proposition 2.5]{Buzyakova} for a
more general statement). Thus $CH(\mathscr{O})=CH(\mathscr{S})$.
\end{proof}

Since $\mathbf{SPsRad}$ is a coreflective hull of a class of compact
Hausdorff spaces, we are motivated to use Theorem 4.4 of \cite{BT} to show
that $\mathbf{SPsRad}$ inherits the structure of a Cartesian closed category
(which is, in fact, a convenient category of topological spaces, in the
sense of \cite{Steenrod}, since it contains the sequential category \cite[%
Proposition 7.3]{BT}). We follow the usual construction of an internal
product and function space for coreflective hulls \cite{BT}.

A subset $A\subset X$ is $\mathscr{S}$-closed if for every non-compact
well-ordered space $\alpha$ and map $f:\alpha\cup \{\infty_{\alpha}\}\to X$
with $f(\alpha)\subset A$, then $f(\infty_{\alpha})\in A$. Observe the $%
\mathscr{S}$-closed sets determine a topology on the underlying set of $X$
which is finer than the topology of $X$. Let $\mathbf{S}X$ denote the
resulting strongly pseudoradial space. The functor $\mathbf{S}:\mathbf{Top}%
\to \mathbf{SPsRad}$ is a coreflection in the sense that it is right adjoint
to the inclusion $\mathbf{SPsRad}\to\mathbf{Top}$.

This construction provides the internal categorical product $X\times_{%
\mathscr{S}}Y=\mathbf{S}(X\times Y)$ for $\mathbf{SPsRad}$. The internal
mapping spaces are constructed as follows. Given strongly pseudoradial
spaces $X$, $Y$, the set of all continuous functions $X\to Y$ is denoted $%
M(X,Y)$. For any non-compact well-ordered space $\alpha$, map $t:\alpha\cup
\{\infty_{\alpha}\}\to X$, and open set $U\subset Y$, let $W(t,U)=\{f\in
M(X,Y)|ft(\alpha\cup \{\infty_{\alpha}\})\subseteq U\}$. The $\mathscr{S}$%
\textit{-open topology} on $M(X,Y)$ is the topology generated by all of the
subbasic sets of the form $W(t,U)$. Let $M_{\mathscr{S}}(X,Y)$ denote $%
M(X,Y) $ with the $\mathscr{S}$-open topology.

It is not necessarily true that $M_{\mathscr{S}}(X,Y)$ is strongly
pseudoradial. Therefore, we take the function space in $\mathbf{SPsRad}$ to
be the coreflection $\mathbf{S}M_{\mathscr{S}}(X,Y)$. We arrive at the main
result of this section: $\mathbf{SPsRad}$ is Cartesian closed coreflective
subcategory of $\mathbf{Top}$. This result is in contrast with the fact that
the category $\mathbf{PsRad}$ of pseudoradial spaces is not Cartesian closed 
\cite{Cincura}. Moreover, since $\mathbf{SPsRad}$ contains the sequential
category, $\mathbf{SPsRad}$ is a ``convenient category" in the sense of \cite%
{Steenrod}.

\begin{theorem}
\label{cartesianclosed} The category $\mathbf{SPsRad}$ with $\mathscr{S}$%
-product $X\times _{\mathscr{S}}Y$ and function space $\mathbf{S}M_{%
\mathscr{S}}(X,Y)$ is Cartesian closed, i.e. for any strongly pseudoradial
spaces $X,Y,Z$, there is a natural homeomorphism 
\begin{equation*}
\mathbf{S}M_{\mathscr{S}}(X,\mathbf{S}M_{\mathscr{S}}(Y,Z))\cong \mathbf{S}%
(X\times _{\mathscr{S}}Y,Z).
\end{equation*}
\end{theorem}

\begin{proof}
The theorem follows directly from Theorem 4.4 of \cite{BT} once the
following two conditions are verified: 1. for each $\alpha\cup\{\infty_{%
\alpha}\}$, $\beta\cup\{\infty_{\beta}\}\in \mathscr{S}$, the direct product 
$(\alpha\cup\{\infty_{\alpha}\})\times (\beta\cup\{\infty_{\beta}\})$ is
strongly pseudoradial and 2. $\mathscr{S}$ is a regular class of spaces \cite%
[Definition 2.2]{BT} ($\mathscr{S}$ is regular if for each element $%
\gamma\in S$ with $S\in \mathscr{S}$, every neighborhood $U$ of $\gamma$ in $%
S$ contains a closed neighborhood $C$ for which there is a surjection $%
s:B\rightarrow C$ with $B\in \mathscr{S}$.

\begin{enumerate}
\item Suppose $\alpha\cup \{\infty _{\alpha}\}$, $\beta\cup \{\infty
_{\beta}\}\in \mathscr{A}$ and without loss of generality that $\beta<\alpha$%
. Since $(\alpha\cup \{\infty _{\alpha}\})\times (\beta\cup \{\infty
_{\beta}\})$ is a compact, it is a closed subset of $(\alpha\cup \{\infty
_{\alpha}\})^2$. Since $(\alpha\cup \{\infty _{\alpha}\})^2$ is strongly
pseudoradial (See Lemma \ref{jtimesj} below) and $\mathbf{SPsRad}$ is closed
under taking closed subsets, $(\alpha\cup \{\infty _{\alpha}\})\times
(\beta\cup \{\infty _{\beta}\})$ is strongly pseudoradial.

\item Suppose $\alpha\cup \{\infty _{\alpha}\}\in \mathscr{S}$. If $%
\gamma\in \alpha\cup \{\infty _{\alpha}\}$ is an isolated point, we set $%
B=\alpha\cup \{\infty _{\alpha}\}$ and the constant map $s:B\to C=\{\gamma\}$
suffices. If $\gamma$ is a limit point of $\alpha\cup \{\infty _{\alpha}\}$,
we may assume $U=(\gamma_{0},\gamma]$ for $\gamma_{0}<\gamma$. For any $%
\gamma_0<\gamma ^{\prime}<\gamma$, we have $[\gamma
^{\prime},\gamma]=[\gamma ^{\prime},\gamma)\cup \{\gamma\}\in \mathscr{S}$.
Thus we set $B=C=[\gamma ^{\prime},\gamma]$ and let $s:B\to C$ be the
identity map.
\end{enumerate}
\end{proof}

The main difficulty in the proof of the above theorem is verifying that the
product of two well-ordered spaces is strongly pseudoradial. The following
technical lemma is required for the proof of Lemma \ref{jtimesj}.

\begin{lemma}
\label{ketjj}Suppose $K$ is a compact, well-ordered space with minimal
element $0.$ Suppose $(M,m)\in K\times K.$ Suppose $B\subset \lbrack
0,M]\times \lbrack 0,m]$ such that for all $k<M,$ $([0,k]\times \lbrack
0,m])\cap B$ is closed in $[0,M]\times \lbrack 0,m]$ and such that $%
\emptyset =(\{M\}\times \lbrack 0,m])\cap B=([0,M]\times \{m\})\cap B.$
Suppose $(M,m)\in \overline{B}\backslash B$. Then there exists a limit point 
$l\in \lbrack 0,M]$ and a map $f:[0,l]\rightarrow \lbrack 0,M]\times \lbrack
0,m]$ such that $f([0,l))\subset B$ and $f(l)=(M,m).$
\end{lemma}

\begin{proof}
We will define $f:[0,l]\rightarrow [0,M]\times [0,m]$ so that if $%
f(k)=(x(k),y(k))$ then each of the maps $x$ and $y$ is strictly increasing.

To achieve this, at each stage of the definition of $f$ we make as little
strict progress in the direction of $[0,M]$ as possible, while guaranteeing
positive progress in the direction of $[0,m].$ Thus, to implement the
transfinite recursive definition of $f,$ if $k$ is not a limit point of $%
[0,M],$ (and working within $B$) $f(k)$ is defined so that ``starting at $%
f(k-1)$ we move our current abscissa as little as possible strictly to the
right subject to the demand that strict vertical progress is possible at the
new abscissa. Then, having selected our new abscissa, we then claim as much
vertical progress as possible. If $k$ is a limit point of $[0,M]$ then
continuity of $f|_{[0,k)}$ (and compactness of $[0,M]\times \lbrack 0,m]$)
forces the definition of $f(k)$ to be the unique value such that $f|_{[0,k]}$
is continuous.

Before defining $f,$ we build a few basic observations following directly
from our hypotheses and previous observations.

Observation 0: For all $k<M,$ $(\{k\}\times \lbrack 0,m])\cap B$ is closed
in $[0,M]\times \lbrack 0,m].$

Observation 1: $M$ is a limit point of $[0,M].$ ( To obtain a contradiction,
suppose otherwise. Then $\{M\}$ is open in $[0,M],$ and hence, since $(M,m)$
is a limit point of $B,$ the open set $\{M\}\times \lbrack 0,m]$ contains a
point $(M,y_{m})\in B$ such that $(M,y_{m})\neq (M,m),$ contradicting the
hypothesis that $\emptyset =(\{M\}\times \lbrack 0,m])\cap B).$

Observation 2: By a symmetric argument applied to Observation 1, $m$ is a
limit point of $[0,m].$

Observation 3: By Observations 1 and 2, basic open sets $U\times V$ of $%
[0,M]\times \lbrack 0,m]\ $containing $(M,m)$ are of the form $(a,M]\times
(b,m]$ with $a<M$ and $b<m.$

Observation 4: If $W$ is an open set of $[0,M]\times \lbrack 0,m]$ such that 
$(M,m)\in W$ then (since $(M,m)$ is a limit point of $B$) Observation 3
ensures there exists $(x,y)\in W\cap B$ such that $x<M$ and $y<m$ and also,
for each $(x,y)\in W\cap B$ there exists $(x^{\ast },y^{\ast })\in W\cap B$
such that $x<x^{\ast }<M$ and $y<y^{\ast }<m.$

For each $k\in \lbrack 0,M]$ define $B_{k}=(\{k\}\times \lbrack 0,m])\cap B$%
. If $B_{k}\neq \emptyset $ (by Observation 0) let $m_{k}$ be minimal such
that $B_{k}\subset \{k\}\times \lbrack 0,m_{k}]$ (Thus $(k,m_{k})$ is the
`maximal' element of $B_{k}$).

Note $B\neq \emptyset $ since $\overline{B}\neq \emptyset .$ Obtain $%
x_{0}\in \lbrack 0,M]$ minimal such that $B_{x_{0}}\neq \emptyset .$ Define $%
f(0)=(x_{0},m_{0})$ and let $y_{0}=m_{0}.$ By hypothesis $x_{0}<M$ and $%
y_{0}<m.$

Suppose $k\in \lbrack 0,M]$ and $f(i)\in B$ for all $i<k$ so that all of the
following hold:

i) If $i<k$ then $f(i)=(x_{i},y_{i})\in B.$

ii) If $i<k$ then $i\leq y_{i}$ and $i\leq x_{i}.$

iii) If $i<j<k$ then $x_{i}<x_{j}<M$ and $y_{i}<y_{j}<m.$

iv) $f|_{[0,k)}$ is continuous.

If $k-1$ exists obtain $x_{k}\in \lbrack 0,M]$ minimal such that $%
x_{k-1}<x_{k}<M$, $B_{k}\neq \emptyset $, and $y_{k-1}<m_{k}<m$. Define $%
f(k)=(x_{k},m_{k})$ and let $y_{k}=m_{k}.$

To see that $f(k)$ is well defined let $W=(x_{k-1},M]\times (y_{k-1},m]$.
Observation $4$ ensures the existence of the desired $f(k).$ Conditions i)
and iii) are preserved by definition. To check condition ii), let $i=k-1.$
Thus $k-1\leq x_{k-1}$ and $k-1\leq y_{k-1}.$ Thus $(k-1)+1\leq x_{k}$ and $%
(k-1)+1\leq y_{k}.$ For condition iv), notice $[0,k+1)=[0,k)\cup \{k\}$ and $%
[0,k)=[0,k-1].$ Thus $[0,k-1)$ is the union of two disjoint closed sets and
hence continuity of $f|_{[0,k+1)}$ follows from the familiar pasting from
general topology.

If $k-1$ does not exist then $k$ is a limit point of $K,$ and define $%
f(k)=(\sup_{i<k}\{x_{i}\},\sup_{i<k}\{y_{i}\})=(x_{k},y_{k}).$ Condition
iii) and the l.u.b. property of $K$ \ ensure $f(k)$ is well defined.

To check continuity of $f|_{[0,k]}$ suppose $W$ is a basic open set of $%
[0,M]\times \lbrack 0,m].$ Let $U=f|_{[0,k]}^{-1}(W).$ To check iv) If $%
(M,m)\notin W$ then continuity of $f|_{[0,k)}$ ensures $U$ is open in $[0,k)$
and, (since $[0,k)$ is open in $[0,k]$), $U$ is open in $[0,k].$ Suppose $%
(M,m)\in W.$ Then (since $f|_{[0,k)}$ is increasing) there exists $a\in
\lbrack 0,M)$ such that $f_{[0,k)}^{-1}(W)=(a,k).$ Thus $%
f_{[0,k]}^{-1}(W)=(a,k)\cup \{k\}=(a,k]$. Since $(a,k]$ is open in $[0,k],$
we conclude $f|_{[0,k]}$ is continuous. By definition $%
f|_{[0,k+1)}=f|_{[0,k]}$ and thus condition iv) is preserved since $%
f|_{[0,k+1)}$ is continuous.

Condition ii) for $i<k$ combined with continuity of $f|_{[0,k]}$ ensure
condition ii) is preserved for $i\leq k+1.$

By definition $f(k)=(x_{k},y_{k}).$ Preservation of the remaining conditions
depend on whether $f(k)\in B$ or not.

Case 1. If $f(k)\in B$ then by hypothesis of the Lemma, $x_{i}<M$ and $%
y_{i}<m.$ Moreover, since $\{x_{i}\}$ and $\{y_{i}\}$ for $i\leq k$ are
strictly transfinite sequences, condition iii) is preserved.

Case 2. Suppose $f(k)\notin B.$ Then, by continuity of $f,$ $f(k)$ is a
limit point of $B$.

Observation 5 ensures $(M,m)$ is the only limit point of $B,$ and hence $%
f(k)=(M,m)$. Let $l=k$ and the Lemma at hand is proved.
\end{proof}

\begin{lemma}
\label{jtimesj} If $\alpha $ is a non-compact well-ordered space, then the
product $(\alpha\cup\{\infty_{\alpha}\})\times
(\alpha\cup\{\infty_{\alpha}\})$ (with the standard product topology) is
strongly pseudoradial.
\end{lemma}

\begin{proof}
Suppose $A$ is a non-closed subset of $(\alpha\cup\{\infty_{\alpha}\})\times
(\alpha\cup\{\infty_{\alpha}\})$.

We seek a well-ordered subspace $\beta\subset \alpha$ such that $\overline{%
\beta}\backslash \beta=\{\infty_{\beta}\}$ contains only the minimum element
of $\{\gamma\in \alpha|\gamma>\beta\}$ and a map $f:\beta\cup\{\infty_{%
\beta}\}\to (\alpha\cup\{\infty_{\alpha}\})\times
(\alpha\cup\{\infty_{\alpha}\})$ such that $f(\beta)\subset A$ and $f(\infty
_{\beta})\notin A.$

First we reduce as follows to the case that each ``vertical or horizontal
slice'' of $A$ is closed in $(\alpha\cup\{\infty_{\alpha}\})\times
(\alpha\cup\{\infty_{\alpha}\}).$ For each $x\in
\alpha\cup\{\infty_{\alpha}\}$ define $B_{x}=A\cap (\{x\}\times
(\alpha\cup\{\infty_{\alpha}\})).$ Note each subspace $\{x\}\times
(\alpha\cup\{\infty_{\alpha}\})$ is closed in $(\alpha\cup\{\infty_{\alpha}%
\})\times (\alpha\cup\{\infty_{\alpha}\})$ and is also canonically
homeomorphic to the strongly pseudoradial space $\alpha\cup\{\infty_{\alpha}%
\}.$ Thus, if there exists $x\in \alpha\cup\{\infty_{\alpha}\}$ such that $%
B_{x}$ is not closed in $(\alpha\cup\{\infty_{\alpha}\})\times
(\alpha\cup\{\infty_{\alpha}\})$, then there exists a non-compact
well-ordered subspace $\beta\subset \alpha\cup\{\infty_{\alpha}\}$ and a map 
$f:\beta\cup \{\infty _{\beta}\}\rightarrow \{x\}\times
(\alpha\cup\{\infty_{\alpha}\})$ such that $f(\beta)\subset B_{x}$ and $%
f(\infty _{\beta})\notin B_{x}.$ Since $\{x\}\times
(\alpha\cup\{\infty_{\alpha}\})$ is closed in $(\alpha\cup\{\infty_{\alpha}%
\})\times (\alpha\cup\{\infty_{\alpha}\})$, it follows that $f(\infty
_{\beta})\in \{x\}\times (\alpha\cup\{\infty_{\alpha}\}).$ Hence $f(\infty
_{\beta})\notin A$ and we have the desired map $f.$

After applying a symmetric argument to slices of the form $%
(\alpha\cup\{\infty_{\alpha}\})\times \{y\}$ we have reduced to the case
that the subspaces $(\{x\}\times (\alpha\cup\{\infty_{\alpha}\}))\cap A$ and 
$(\alpha\cup\{\infty_{\alpha}\})\times \{y\})\cap A$ are closed in $%
(\alpha\cup\{\infty_{\alpha}\})\times (\alpha\cup\{\infty_{\alpha}\})$ for
all $\{x,y\}\subset (\alpha\cup\{\infty_{\alpha}\}).$

Let $K$ be a compact set. Note $(\{0\}\times K)\cap A$ is closed and $%
(K\times K)\cap A$ is not closed. Hence there exists $M\in K$ minimal such
that $([0,M]\times K)\cap A$ is not closed. By our assumptions $([0,M]\times
\{0\})\cap A$ is closed. Thus there exists $m\in K$ minimal such that $%
([0,M]\times \lbrack 0,m])\cap A$ is not closed in $K\times K.$

Let $C=([0,M]\times \lbrack 0,m])\cap A.$ Since $[0,M]\times \lbrack 0,m]$
is closed in $K\times K,$ $C$ is not closed in $[0,M]\times \lbrack 0,m].$
Hence $\overline{C}\backslash C\neq \emptyset $ and $\overline{C}\backslash
C\subset \lbrack 0,M]\times \lbrack 0,m].$ Suppose $(x,y)\in \overline{C}%
\backslash C.$ To obtain a contradiction suppose $x<M.$ Then $(x,y)$ is a
limit point of the closed set $([0,x]\times \lbrack 0,m])\cap A$ and hence $%
(x,y)\in A,$ a contradiction. Thus $x=M.$ Suppose, to obtain a contradiction 
$y<m.$ Then $(M,y)$ is a limit point of the closed set $([0,M]\times \lbrack
0,y])\cap A$ and hence $(M,y)\in A,$ a contradiction. Hence $\overline{C}%
\backslash C=\{(M,m)\}.$

Note $(M,m)$ is not in the closed set $(([0,M]\times \{m\})\cup (\{M\}\times
\lbrack 0,m]))\cap C.$ Obtain $a\in \lbrack 0,M)$ and $b\in \lbrack 0,m)$
such that $([a+1,M]\times \{m\})\cap C=(\{M\}\times \lbrack b+1,m])\cap
C=\emptyset .$ Note $[a+1,M]\times \lbrack b+1,m]$ is clopen in $[0,M]\times
\lbrack 0,m].$ Let $B=[a+1,M]\times \lbrack b+1,m]\cap C.$ Note $(K,B,M,m)$
satisfy the hypothesis of Lemma \ref{ketjj} and obtain a limit point $l\in
\lbrack 0,M]$ and a map $f:[0,l]\rightarrow \overline{B}$ such that $%
f([0,l))\subset B\subset A$ and $f(l)\notin B.$ Let $\beta =[0,l)$ and $%
\infty _{\beta }=l.$ By definition $\overline{B}\backslash B\subset 
\overline{C\backslash }C\subset \overline{A}\backslash A$ and hence $f(l)\in
A$.
\end{proof}

\section{Compactness and related properties in $\mathbf{SPsRad}$}

\label{compactsection}

We seek to generalize Proposition \ref{Franklinprop} below, a basic fact in 
\textbf{SEQ. }Our main result is Theorem \ref{compactthm}. Lemma \ref{adjust}
shows Theorem \ref{compactthm} is a generalization of Proposition \ref%
{Franklinprop}, provided we restrict our attention to so-called \textit{UW} spaces. Corollary \ref{compactcor} generalizes in $%
\mathbf{SPsRad}$ the fact that compact weakly Hausdorff sequential spaces
are sequentially compact.

\begin{proposition}
\label{Franklinprop} If $X$ is a sequential space then $X$ is countably
compact iff $X$ is sequentially compact \cite{Franklin1}.
\end{proposition}

The following strengthening of ``sequentially compact'' is strict as shown
by the minimal uncountable well ordered space.

\begin{definition}
The space $X$ is \textbf{strongly pseudoradially compact} if for each
noncompact well ordered space $\alpha $ and each map $f:\alpha \rightarrow
X, $ there exists $\beta $ closed and cofinal in $\alpha $ so that $f|\beta $
is continuously extendable at $\infty _{\alpha }.$
\end{definition}

Corollary \ref{compactcor} shows every compact weakly Hausdorff space $X\in 
\mathbf{SPsRad}$ is strongly pseudoradially compact, but the proof exploits
the fact that maps of compact ordinals into $X$ have closed image. To absorb
various difficulties created when $X$ is not weakly Hausdorff we adjust our
definitions as follows.

\begin{definition}
Suppose $X$ is a space and $\alpha $ is a noncompact well ordered space. A
map $f:\alpha \rightarrow X$ is \textbf{decent} if there exists a map $%
g_{f}:\alpha \rightarrow X$ so that $g(i)\in \overline{\{f(i)\}}$ for all $%
i\in \alpha ,$ and $g(C)$ is closed for all compact $C\subset \alpha .$
\end{definition}

\begin{definition}
If $X$ is a space declare $X$ \textbf{decently strongly pseudoradially
compact} if for each noncompact well ordered space $\alpha $ and each decent
map $f:\alpha \rightarrow X,$ there exists $\beta $ closed and cofinal in $%
\alpha $ so that $f|\beta $ is continuously extendable at $\infty _{\alpha
}. $
\end{definition}

Equivalent to the standard finite open subcover formulation, recall a space $%
X$ is compact iff $\emptyset \neq \cap _{i\in I}A_{i}$ for all collections
of closed sets $\{A_{i}\}$ with the finite intersection property. Weaker
than compact, the closed set formulation of ``countably compact'' is $%
\emptyset \neq \cap _{n=1}^{\infty }L_{n}$ for each nested sequence of
closed sets $...L_{3}\subset L_{2}\subset L_{1},$ and this motivates the
following definition.

\begin{definition}
The space $X$ is \textbf{decently compact} if for each decent map $f:\alpha
\rightarrow X$ then $\emptyset \neq \cap _{i\in \alpha }L_{i}$ if $L_{i}=%
\overline{\cup _{j=i}^{\infty _{\alpha }}f(j)}$.
\end{definition}

We observe basic facts about decent maps and decently compact spaces as
follows.

\begin{remark}
\label{decentrem}Suppose the map $g_{f}:\alpha \rightarrow X$ shows the map $%
f:\alpha \rightarrow X$ is decent. Then $\{g(i)\}$ is closed in $X$, since $%
\{i\}$ is compact. If $X$ is $T_{1}$ then $g_{f}=f$ since $\{f(i)\}=%
\overline{\{f(i)\}}.$ If $X$ is weakly Hausdorff then each map $f:\alpha
\rightarrow X$ is decent. If $X$ is compact then $X$ is decently compact. If 
$X$ is a space and the map $g_{f}:\alpha \rightarrow X$ shows $f$ is decent,
then $\cap _{i\in \alpha }(\overline{\cup _{j=i}^{\infty _{\alpha }}g_{f}(j)}%
)\subset \cap _{i\in \alpha }(\overline{\cup _{j=i}^{\infty _{\alpha }}f(j)}%
),$ (since $\cup _{j=i}^{\infty _{\alpha }}g_{f}(j)\subset \cup
_{j=i}^{\infty _{\alpha }}\overline{\{f(j)\}}\subset \overline{\cup
_{j=i}^{\infty _{\alpha }}f(j)}$). Since $g_{f}$ is decent, the space $X$ is
decently compact iff $\emptyset \neq \cap _{i\in \alpha }(\overline{\cup
_{j=i}^{\infty _{\alpha }}f(j)})$ for all decent maps $f:\alpha \rightarrow
X $ such that $\overline{\{f(i)\}}=\{f(i)\}$ for all $i\in \alpha .$
\end{remark}

The following elementary Lemma is used in the proof of Lemma \ref{tee1}.

\begin{lemma}
\label{contcheck}Suppose $X\in \mathbf{SPsRad}$, $Y$ is a space and $%
f:X\rightarrow Y$ is a function. Then $f$ is continuous if and only if $%
f\circ g$ is continuous whenever $\alpha $ is a noncompact well ordered
space, and $g:\alpha \cup \{\infty _{\alpha }\}\rightarrow X$ is a map.
\end{lemma}

\begin{proof}
If $f$ is continuous then the composition of maps $f\circ g$ is continuous.
If $f$ is not continuous obtain $A\subset Y$ closed so that $B=f^{-1}(A)$ is
not closed. Obtain a well ordered space $\alpha $ and map $g:\alpha \cup
\{\infty _{\alpha }\}$ so that $g(\alpha )\subset B$ and $g(\infty _{\alpha
})\notin B.$ Then $f\circ g$ is not continuous, since if $f\circ g$ were
continuous we obtain the contradiction $(f\circ g)^{-1}(A)=\alpha \cup
\{\infty _{\alpha }\}.$
\end{proof}

\begin{lemma}
\label{tee1}Suppose $X$ is strongly pseudoradial, and decently compact.
Suppose $f:\alpha \rightarrow X$ is decent and $f(i)=\overline{\{f(i)\}}$
for all $i\in \alpha .$ Then there exists $\beta $ closed and cofinal in $%
\alpha ,$ so that $f|\beta $ is continuously extendable to $f|\beta \cup
\{\infty _{\alpha }\}.$
\end{lemma}

\begin{proof}
Note if $g_{f}$ shows $f$ is decent then $g_{f}=f.$ In particular $f(C)$ is
closed for all compact $C\subset X.$ If $f^{-1}(x)$ is unbounded for some $%
x\in X,$ then $x=\overline{\{x\}},$ and letting $\beta =f^{-1}(x),$ we have $%
\beta $ closed and cofinal in $\alpha ,$ and $f|\beta $ is continuously
extendable mapping $\infty _{\beta }\rightarrow x.$ If $f^{-1}(x)$ is
bounded for all $x$ then by Lemmas \ref{mono} and \ref{reg} there exists a
regular cardinal $\alpha _{1}$ closed and cofinal in $\alpha $ so that $%
f|\alpha _{1}$ is one to one. Since $f|\alpha _{1}$ is decent, we may assume
wolog that $\alpha $ is a regular cardinal and $f$ is one to one. Observe $%
\emptyset =\cap _{i\in \alpha }f([i,\infty _{\alpha }))$ since $f$ is one to
one. To obtain a contradiction suppose there exists $\beta $ closed and
cofinal in $\alpha $ so that $f(\beta \cap \lbrack i,\infty _{\alpha }))$ is
closed in $X$ for all $i\in \alpha .$ Then $\emptyset =\cap _{i\in \alpha
}f(\beta \cap \lbrack i,\infty _{\alpha })),$ contradicting the fact that $X$
is decently compact. Thus there exists $i\in \alpha $ so that $f[i,\infty
_{\alpha })$ is not closed. Since $[i,\infty _{\alpha })$ is closed and
cofinal in $\alpha ,$ the subspace $[i,\infty _{\alpha })$ is a regular
cardinal and thus wolog we may assume $f(\alpha )$ is not closed in $X.$

Since $f(\alpha )$ is not closed in $X$, by Theorem \ref{sprchar}, obtain a
noncompact regular cardinal $\gamma $ and an extendable injective map $%
g:\gamma \rightarrow f(\alpha )$ so that $g(\infty _{\gamma })\notin
f(\alpha ).$ Define $h:\gamma \rightarrow \alpha $ via $h=f^{-1}\circ g.$
Let $\beta =h(\gamma ).$ The plan is to show $h$ is a closed homeomorphism
and $\beta $ is cofinal in $\alpha .$

Since $h$ is one-to-one, $\left\vert \gamma \right\vert \leq \left\vert
\alpha \right\vert .$ Thus, since $\gamma $ is a regular cardinal if $%
C\subset \gamma $ is bounded, $\left\vert C\right\vert <\left\vert \gamma
\right\vert \leq \left\vert \alpha \right\vert $ and in particular $h(C)$ is
bounded in $\alpha .$

To show that $h$ is continuous suppose $\kappa $ is a noncompact well
ordered space and $p:\kappa \cup \{\infty _{\kappa }\}\rightarrow \gamma $
is a map. Then $im(p)$ is compact and thus bounded. Thus, by the previous
paragraph, $im(h(p))$ is bounded in $\alpha .$ Obtain a compact $K\subset
\alpha $ so that $im(h(p))\subset K.$ By hypothesis $f|K$ is a closed
embedding. Thus $f^{-1}|f(K)$ is continuous. Hence $f^{-1}g\kappa $ is
continuous. Thus by Lemma \ref{contcheck} $h$ is continuous.

Suppose $A\subset \gamma $ and $B=h(A).$ If $A$ is compact then (since $h$
is continuous), $B$ is compact and hence $B$ is closed. If $A$ is unbounded
then $g(\infty _{\gamma })$ is a limit point of $g(A)$ and hence $g(\infty
_{\gamma })$ is a limit point of $f(B)=f(f^{-1}(g(A))).$ Thus $B$ is
unbounded since otherwise we obtain the contradiction $g(\infty _{\gamma
})\in f(\overline{B}).$ If $A$ is closed and unbounded, to see that $B$ is
closed, suppose otherwise and obtain minimal $l\in \overline{B}\backslash B.$
Let $B_{1}=[0,l]$ and $A_{1}=h^{-1}(B_{1}).$ Then $A_{1}$ is unbounded since
otherwise we obtain the contradiction $l\in h(A_{1})\subset B_{1}.$ Thus,
applying the same argument as shown above, we deduce $g(\infty _{\gamma })$
is a limit point of $g(A_{1})$ and obtain the contradiction $B_{1}$ is
unbounded.

Since $h:\gamma \rightarrow \beta $ is a homeomorphism $f|\beta $ is
continuously extendable, with $\infty _{\alpha }\rightarrow g(\infty
_{\gamma }).$
\end{proof}

\begin{lemma}
\label{extend}Suppose $\alpha $ is an unbounded well ordered space and $%
g:\alpha \cup \{\infty _{\alpha }\}\rightarrow X$ is a map such that $%
\{g(i)\}$ is closed in $X$ for all $i\in \alpha .$ Suppose $f:\alpha
\rightarrow X$ is a map such that $g(i)\in \overline{\{f(i)\}}$ for all $%
i\in \alpha .$ Then $F=f\cup \{\infty _{\alpha },g(\infty _{\alpha })\}$ is
continuous.
\end{lemma}

\begin{proof}
By hypothesis $F|\alpha $ is continuous. To check $F$ is continuous at $%
\infty _{\alpha }$ suppose $U$ is open in $X$ and $g(\infty _{\alpha })\in
U. $ By continuity of $g,$ obtain $K$ so that $g(i)\in U$ if $K\leq i.$ Thus 
$f(i)\in U$ if $K\leq i.$
\end{proof}

\begin{theorem}
\label{compactthm} If $X$ is strongly pseudoradial, then the following are
equivalent:
\end{theorem}

\begin{enumerate}
\item $X$ is decently strongly pseudoradially compact,

\item $X$ is decently compact.
\end{enumerate}

\begin{proof}
Suppose $f:\alpha \rightarrow X$ is decent as shown by the map $g_{f}.$

$(1)\Rightarrow (2)$ By (1), obtain $\beta $ closed and cofinal in $\alpha $
and obtain $y\in X$ so that so that $G=g_{f}|\beta \cup \{(\infty _{\alpha
},y)\}$ is continuous. Since $G$ is continuous at $\infty _{\alpha }$, $y\in 
\overline{g_{f}([i,\infty _{\alpha }))}$ for all $i\in \alpha .$ By Remark %
\ref{decentrem}, $y\in \overline{f([i,\infty _{\alpha }))}$ for all $i\in
\alpha .$ Thus $X$ is decently compact.

(2) $\Rightarrow $ (1) By Lemma \ref{tee1} obtain $\beta $ closed and
cofinal in $\alpha $ and obtain $y\in X$ so that so that $G=g_{f}|\beta \cup
\{(\infty _{\alpha },y)\}$ is continuous. By Lemma \ref{extend} $f|\beta
\cup \{\infty _{\alpha },y\}$ is continuous.
\end{proof}

\begin{corollary}
\label{compactcor}Suppose $X$ is strongly pseudoradial and weakly Hausdorff.
If $X$ is compact then $X$ is strongly pseudoradially compact.
\end{corollary}

\begin{proof}
Since $X$ is compact, $X$ is decently compact. Since $X$ is weakly Hausdorff
each map $f:\alpha \rightarrow X$ is decent. Now apply theorem \ref%
{compactthm}.
\end{proof}

\begin{example}
The main example \cite{OSt} of Ostaszewski shows the converse of Corollary %
\label{spc} is false. If $X$ is strongly pseudoradial, weakly Hausdorff and
strongly pseudoradially compact, then $X$ need not be compact.
\end{example}

\begin{proof}
The main space $X$ in \cite{OSt} is Hausdorff, sequential, sequentially
compact, but not compact. Since $X$ is sequential, $X$ is strongly
pseudoradial. To check that $X$ is strongly pseudoradially compact, suppose $%
\alpha $ is an unbounded well ordered space and $f:\alpha \rightarrow X$ is
a map. If $\alpha $ is countable obtain a sequence $\beta $ cofinal in $%
\alpha $ and apply sequential compactness to obtain a continuous extension
of $f|\beta .$ Suppose $\alpha $ is uncountable. If $f^{-1}(x)$ is unbounded
for some $x,$ let $\beta =f^{-1}(x)$ and obtain the extension $\infty
_{\alpha }\rightarrow x$. To obtain a contradiction suppose $f^{-1}(x)$ is
bounded for all $x.$ By Lemma \ref{mono} obtain $\beta $ closed and cofinal
in $\alpha $ so that $f|\beta $ is one to one. Since $X$ is Hausdorff, $f|C$
is an embedding for all compact $C\subset \beta .$ Thus $f(\beta )$ is
closed in $X$ (since if $f(\beta )$ is not closed in the sequential space $%
X, $ we obtain a cofinal sequence in the uncountable well ordered set $\beta 
$). Thus $f$ embeds $\beta $ onto a closed subspace $f(\beta )\subset X.$
Obtain a closed subspace $A\subset f(\beta )$ so that $A$ and $f(\beta
)\backslash A$ are uncountable. Then $X\backslash A$ is open. However both $%
X\backslash A$ and $A$ are uncountable, contradicting the fact (Lemma 1.3 
\cite{OSt}) that in $X,$ every open subspace is countable or cocountable.
\end{proof}

Unfortunately Theorem \ref{compactthm} is not a generalization of
Proposition \ref{Franklinprop} as shown by the following example, in which $%
X $ is vacuously decently compact since no map $f:\alpha \rightarrow X$ is
decent.

\begin{example}
\label{t0bad}Consider the countable set $X=\{1,2,3,...\}$ with topology
generated by the closed sets $[n,\infty )$. Then $X$ is sequential and
decently compact but $X$ is not sequentially compact.
\end{example}

Lemma \ref{adjust} effectively shows the phenomenon in Example \ref{t0bad}
is the only obstruction preventing a sequential decently compact space from
being sequentially compact. Thus Theorem \ref{compactthm} generalizes
Proposition \ref{Franklinprop} provided we restrict ourselves to spaces with
the following useful weak property, which we call the \textbf{UW property}
in the paper at hand.

\begin{definition}
The space $X$ has the UW \textbf{property }if for each $x\in X$ there exists 
$y\in X$ so that $y\in \overline{\{x\}}$ and $\overline{\{y\}}=\{y\}.$
\end{definition}

\begin{lemma}
If $X$ is compact or $T_{1},$ then $X$ has the UW property.
\end{lemma}

\begin{proof}
Suppose $x\in X.$ Obtain a maximal transfinite sequence (indexed by a well
ordered set) $x=x_{0},x_{1},x_{2},....$ so that $x_{i}\in \overline{\{x_{j}\}%
}\backslash \{x_{j}\}$ for all $j<i.$ Let $L_{i}=\overline{\{x_{i}\}}.$ If $%
X $ is compact $\emptyset \neq \cap L_{i}.$ Let $y\in \cap L_{i}.$ Note $%
y\in \overline{\{x\}},$ and by maximality $\overline{\{y\}}=\{y\}$ (since if
there exists $z_{j}\in \overline{\{y\}}\backslash \{y\}$ we obtain the
contradiction $y\notin L_{j}$). If $X$ is $T_{1}$ let $y=x.$
\end{proof}

\begin{lemma}
\label{adjust}Suppose the sequential space $X$ has the UW property. Then $X$
is sequentially compact iff $X$ is decently compact.
\end{lemma}

\begin{proof}
Suppose $X$ is sequentially compact and $f:\alpha \rightarrow X$ is decent
as shown by the map $g_{f}:\alpha \rightarrow X$. If there exists $x\in X$
so that $g_{f}^{-1}(x)$ is unbounded let $\beta _{1}=g_{f}^{-1}(x)$ and
observe $g_{f}|\beta _{1}$ is continuously extendable at $\infty _{\alpha }$%
. If no such $x$ exists obtain by Lemma \ref{mono} $\beta $ closed and
cofinal in $\alpha $ so that $g_{f}|\beta $ is one to one. By Lemma \ref{reg}
we may assume $\beta $ is a regular cardinal. If $\beta $ is countable then $%
\beta $ has the discrete topology. Since $X$ is sequentially compact there
exists $\beta _{1}$ closed and cofinal in $\beta $ so that $g_{f}|\beta _{1}$
is continuously extendable at $\infty _{\alpha }$. If $\beta _{1}$ is
uncountable we obtain a contradiction as follows. Since $X$ is sequential,
since $g_{f}|C$ is closed for all compact $C\subset \beta _{1},$ and since $%
\beta _{1}$ is uncountable, $g_{f}(\beta _{2})$ is closed for all $\beta
_{2} $ closed and cofinal in $\beta _{1}.$ In particular the uncountable
well ordered space $\beta _{1}$ is homeomorphic to the sequential space $%
g_{f}(\beta _{1})$ and we have a contradiction. Thus by Theorem \ref%
{compactthm} $X$ is decently compact.

Conversely suppose $X$ is decently compact. Consider the minimal infinite
well ordered space $\alpha =\{1,2,3....\}$ and suppose $f:\alpha \rightarrow
X$ is a map. Since $X$ has the UW property, obtain for each $n\in \alpha ,$ $%
y_{n}\in \overline{\{f(n)\}}$ so that $\{y_{n}\}=\overline{\{y_{n}\}}.$ Let $%
g_{f}(n)=y_{n}.$ The map $g_{f}$ shows $f$ is decent. Thus by Theorem \ref%
{compactthm} there exists $\beta $ closed and cofinal in $\alpha $ so that $%
f|\beta $ is continuously extendable at $\infty _{\alpha }.$ Thus $f|\beta $
is the desired convergent subsequence and hence $X$ is sequentially compact.
\end{proof}

We conclude with the following problem.

\begin{problem}
\label{t1prob}Suppose the compact $T_{1}$ space $X\in \mathbf{SPsRad.}$ Must 
$X$ be strongly pseudoradially compact?
\end{problem}

\end{document}